\PassOptionsToPackage{unicode}{hyperref}
\PassOptionsToPackage{hyphens}{url}
\documentclass[
]{article}
\usepackage{lmodern}
\usepackage{amssymb,amsmath,amsthm}
\usepackage{ifxetex,ifluatex}
\ifnum 0\ifxetex 1\fi\ifluatex 1\fi=0 
  \usepackage[T1]{fontenc}
  \usepackage[utf8]{inputenc}
  \usepackage{textcomp} 
\else 
  \usepackage{unicode-math}
  \defaultfontfeatures{Scale=MatchLowercase}
  \defaultfontfeatures[\rmfamily]{Ligatures=TeX,Scale=1}
\fi
\IfFileExists{upquote.sty}{\usepackage{upquote}}{}
\IfFileExists{microtype.sty}{
  \usepackage[]{microtype}
  \UseMicrotypeSet[protrusion]{basicmath} 
}{}
\makeatletter
\@ifundefined{KOMAClassName}{
  \IfFileExists{parskip.sty}{%
    \usepackage{parskip}
  }{
    \setlength{\parindent}{0pt}
    \setlength{\parskip}{6pt plus 2pt minus 1pt}}
}{
  \KOMAoptions{parskip=half}}
\makeatother
\usepackage{xcolor}
\IfFileExists{xurl.sty}{\usepackage{xurl}}{} 
\IfFileExists{bookmark.sty}{\usepackage{bookmark}}{\usepackage{hyperref}}
\hypersetup{
  pdftitle={Constructive proof of Herschfeld's Convergence Theorem},
  pdfauthor={Ran Gutin \\ Department of Computer Science; Imperial College London;\\ South Kensington, London SW7 2BU; r.gutin20@imperial.ac.uk},
  hidelinks,
  pdfcreator={LaTeX via pandoc}}
\usepackage{graphicx}
\makeatletter
\def\maxwidth{\ifdim\Gin@nat@width>\linewidth\linewidth\else\Gin@nat@width\fi}
\def\maxheight{\ifdim\Gin@nat@height>\textheight\textheight\else\Gin@nat@height\fi}
\makeatother
\setkeys{Gin}{width=\maxwidth,height=\maxheight,keepaspectratio}
\makeatletter
\def\fps@figure{htbp}
\makeatother
\providecommand{\tightlist}{%
  \setlength{\itemsep}{0pt}\setlength{\parskip}{0pt}}
\usepackage{cite}
\usepackage{amsmath}
\DeclareMathOperator*{\kap}{\text{\Large{$\mathcal K$}}}

\title{Constructive proof of Herschfeld's Convergence Theorem}
\author{Ran Gutin \\ Department of Computer Science; \\Imperial College London; South Kensington, London SW7 2BU; \\ r.gutin20@imperial.ac.uk}
\date{}

\theoremstyle{theorem}
\newtheorem{theorem}{Theorem}
\theoremstyle{lemma}
\newtheorem{lemma}{Lemma}
\theoremstyle{corollary}
\newtheorem{corollary}{Corollary}[lemma]
\theoremstyle{remark}
\newtheorem*{remark}{Remark}
\theoremstyle{definition}
\newtheorem*{definition}{Definition}

\begin{document}
\maketitle
\begin{abstract}
We give a constructive proof of Herschfeld's Convergence Theorem. The proof is based on simple and generalisable insights about concave/convex functions. Explicit convergence bounds are derived. A new special function is defined to enable one of these bounds to be expressed. We also generalise Herschfeld's result to infinite radicals that nest transfinitely many times.


\end{abstract}

\hypertarget{introduction}{%
\section{Introduction}\label{introduction}}

In this paper, we present a constructive proof of Herschfeld's Convergence Theorem. Our formulation differs from Herschfeld's in a few ways: We consider radicals that nest \emph{more than} infinitely many times, as these are essential to the proof; additionally, we formulate the conditions for convergence in such a way that a constructive proof is possible.

Though the result is stated only for square-roots, it is easy to generalise it to all powers in the interval \((0,1)\). We also show how to apply the techniques presented here to some other continued functions, like continued \(\arctan\).

\hypertarget{motivation}{%
\subsection{Motivation}\label{motivation}}

The motivation for writing this paper was to better understand how to ``constructivise'' arguments that appeal to the Monotone Convergence Theorem. An example of that being Herschfeld's original proof of his eponymous theorem.

Our results can also be applied to show convergence of other continued functions, like continued \(\arctan\). By continued functions, we mean the limit (should it exist) of \(f(A + f(B + f(C + \dotsc)))\). If the function \(f\) is monotonically increasing and concave then our methods are applicable, at least to some degree.

It's worth noting that Herschfeld's proof can be expressed in one page, while the argument here is longer and more sophisticated. However, unlike Herschfeld's non-constructive proof, the argument here states precisely when it is possible to estimate the limit of an infinite radical, and also how to measure the error of such an estimate. Herschfeld's proof only provides knowledge of when the limit exists, but not how to find it.

\hypertarget{related-work}{%
\subsection{Related work}\label{related-work}}

A very thorough chronology of results on problems related to infinite radicals and continued \(f\)-functions is provided by \cite{Jones}.

One of the first complete proofs of the necessary and sufficient condition for convergence of an infinite radical was given by \cite{Hersch}, but an earlier proof was given by Paul Wiernsberger in 1904. Some interesting closed-form solutions of special cases were found by Ramanujan among others \cite{ramanujan1911} (see \cite{Jones} for others). Jones studied various generalisations of Herschfeld's convergence result. These generalisations allow for the powers to be arbitrary positive numbers \cite{Jones1995} in the range \((1,\infty)\) or arbitrary negative numbers \cite{Jones2015} .

Continued radicals continue to be an object of study. See for instance \cite{aoki2016}, \cite{doi:10.4169/amer.math.monthly.121.01.045} and \cite{Convergencemonotonicityandinequalitiesofsequencesinvolvingcontinuedpowers}.

\hypertarget{overview-and-strategy}{%
\section{Overview and strategy}\label{overview-and-strategy}}

A note on notation: All numbers are taken to be non-negative real numbers. The symbol \(\phi\) denotes the golden ratio, equal to both \(\frac{1 + \sqrt 5}{2}\) and the infinite radical \(\sqrt{1 + \sqrt{1 + \dots}}\). Also, \(\mathbb R^+\) denotes the set of non-negative real numbers.

The main theorem of this paper is the combination of Theorem \ref{theorem-sufficient} and Theorem \ref{theorem-necessity}. A simplified version of this theorem states:

\begin{quote}
\textbf{Simplified Constructive HCT}: An infinite radical \(\sqrt{a_1 + \sqrt{a_2 + \dots}}\) where \(a_k \geq 0\) converges if and only if for all \(n\) there exists \(\lim_{k\geq n} \sup_{k=n}^\infty a_k^{2^{-k}}\).
\end{quote}

This theorem is simplified because in our actual theorem statement, the ``infinite radicals'' can nest transfinitely many times.

\hypertarget{strategy-for-proving-if}{%
\subsection{Strategy for proving ``if''}\label{strategy-for-proving-if}}

Let \(M_n = \lim_{k\geq n} \sup_{k=n}^\infty a_k^{2^{-k}}\). A lower bound for the infinite radical is now \(M_1\), and an upper bound is \(M_1 \phi\). We can improve these upper and lower bounds by substituting them deeper inside the radical. In other words, an improved lower bound is \(\sqrt{a_1 + M_2^2}\), and an improved upper bound is \(\sqrt{a_2 + M_2^2 \phi}\). Again, by substituting even deeper inside the radical, we can get a better lower bound in the form \(\sqrt{a_1 + \sqrt{a_2 + M^4}}\), and a better upper bound in the form \(\sqrt{a_1 + \sqrt{a_2 + \phi M^4}}\). And so on.

Now we need to estimate the difference between upper and lower bounds. The key observation is the following: Let's say our upper bound is in the form \(\sqrt{7 + \sqrt{3 + u}}\) for some \(u\), and our lower bound is in the form \(\sqrt{7 + \sqrt{3 + l}}\) for some \(l\). Let \(\epsilon\) be the difference. It so happens that:

\[\begin{aligned}
\epsilon &= \sqrt{7 + \sqrt{3 + u}} - \sqrt{7 + \sqrt{3 + l}} \\
&\leq \sqrt{0 + \sqrt{0 + u}} - \sqrt{0 + \sqrt{0 + l}} \\
&\leq u^{1/4} - l^{1/4}
\end{aligned}\]

In other words, by lowering the numbers \(7\) and \(3\) down to zero, we amplify the difference. But in doing so, we also simplify the estimate for the difference. This is ultimately how we prove ``if''.

\hypertarget{strategy-for-proving-only-if}{%
\subsection{Strategy for proving ``only if''}\label{strategy-for-proving-only-if}}

This is essentially in an inverse strategy to proving ``if''. In doing so, we confront transfinite radicals.

Before we describe transfinite radicals, observe that an inverse to the error overstimate given above, i.e. \[\begin{aligned}
\epsilon &= \sqrt{7 + \sqrt{3 + u}} - \sqrt{7 + \sqrt{3 + l}} \\
&\leq \sqrt{0 + \sqrt{0 + u}} - \sqrt{0 + \sqrt{0 + l}} \\
&\leq u^{1/4} - l^{1/4}
\end{aligned}\] could be \[\begin{aligned}
\epsilon &= \sqrt{7 + \sqrt{3 + u}} - \sqrt{7 + \sqrt{3 + l}} \\
&\geq \sqrt{100 + \sqrt{100 + u}} - \sqrt{100 + \sqrt{100 + l}}
\end{aligned}\] In other words, we can drive the \(3\) and \(7\) to \emph{larger numbers} to \emph{understimate} the error.

\hypertarget{transfinite-radicals}{%
\subsection{Transfinite radicals}\label{transfinite-radicals}}

We consider a more general class of nested radicals than Herschfeld and others. These radicals can nest \emph{more than} infinitely many times. All our results are true for this more general class of nested radicals.

We are led to do this because of a particular family of transfinite radicals that occur naturally in our argument. This family is called the \(U\) function, and is roughly equal to \(\sqrt{1 + \sqrt{1 + \dotsb \sqrt{x^{2^\omega}}}}\) where \(\omega\) denotes the first transfinite ordinal. More precisely, \(U(x)\) is the limit of the sequence \((x, \sqrt{1+x^2}, \sqrt{1 + \sqrt{1 + x^4}}, \dotsc)\) which can be interpreted as a transfinite radical for any fixed \(x\).

\hypertarget{absence-of-differentiation}{%
\subsection{Absence of differentiation}\label{absence-of-differentiation}}

The error estimation technique used in this paper contrasts with the one proposed by Herschfeld, in that it doesn't use any differentiation. This is useful because the derivative of any expression that involves nested radicals is quite complicated. For instance, \[\frac{d}{dx}\sqrt{A + \sqrt{B + \sqrt{C + x}}} = \frac 1{8 \sqrt{C + x} \sqrt{B + \sqrt{C + x}} \sqrt{A + \sqrt{B + \sqrt{C + x}}}}.\] A similar problem afflicts all ``continued functions'', not just continued square-roots. We adopt an error estimation strategy which we have already partially described, and is given in full generality by Lemma \ref{lemma-3}.

We demonstrate a simple example of using Lemma \ref{lemma-3} to show constructively that the continued function \(\arctan(A + \arctan(B + \arctan(C + \dotsc)))\) converges whenever \(A, B, C\) etc. are non-negative. We even give a worst-case error bound.

\hypertarget{concave-functions}{%
\section{Concave functions}\label{concave-functions}}

Proofs of Lemmas \ref{lemma-1} and \ref{lemma-2} can be found elsewhere in the literature, but we give proofs for the sake of completeness.

\begin{lemma}
\hypertarget{lemma-1}{%
\label{lemma-1}}

For any concave function \(h\), real number \(x\), and \(\Delta x \in \mathbb R^+\), it holds that \(h(x+\Delta x) - h(x) \leq h(\Delta x) - h(0)\).
\end{lemma}
\begin{proof}
\hypertarget{proof}{%
\label{proof}}
The definition of concavity implies \(h(\Delta x) \geq \frac{x}{x + \Delta x}h(0) + \frac{\Delta x}{x + \Delta x} h(x + \Delta x).\) Subtracting \(h(0)\) from both sides gives \[\begin{aligned}
h(\Delta x) - h(0) \geq \frac{\Delta x}{x + \Delta x}(h(x+\Delta x) - h(0)).\qquad{(1)}
\end{aligned}\] Similarly, the definition of concavity implies \(h(x) \geq \frac{\Delta x}{x + \Delta x}h(0) + \frac{ x}{x + \Delta x} h(x + \Delta x).\) Subtracting \(h(x + \Delta x)\) from both sides, and then negating, gives \[\begin{aligned}
h(x + \Delta x) - h(x) \leq \frac{\Delta x}{x + \Delta x}(h(x + \Delta x) - h(0)). \qquad{(2)}\\
\end{aligned}\] Combining \((1)\) with \((2)\) gives \[h(\Delta x) - h(0) \geq  \frac{\Delta x}{x + \Delta x}(h(x+\Delta x) - h(0)) \geq h(x + \Delta x) - h(x)\] \end{proof}

\begin{corollary}
\label{corollary-1.1}
Given \(u > l\), \(b > s\) and a concave function \(h\), it holds that \(h(b + u) - h(b + l) < h(s + u) - h(s + l)\).
\end{corollary}
\begin{remark}
\(b\), \(s\), \(u\) and \(l\) can be remembered using the mnemonics ``bigger'', ``smaller'', ``upper'' and ``lower'', which indicate through opposite meanings which is greater than which. There is no implied ordering between bigger and lower, or smaller and upper, because bigger isn't opposite to lower.
\end{remark}
\begin{proof}
\hypertarget{proof-1}{%
\label{proof-1}}
Apply Lemma \ref{lemma-1} to \(H(X) = h(s + l + X)\), \(x = b - s\), \(\Delta x = u - l\) to get \(H(\Delta x) - H(0) \geq H(x + \Delta x) - H(x)\) which is equivalent to \(h(s + u) - h(s + l) \geq h(b + u) - h(b + l)\).

\end{proof}
\begin{lemma}
\hypertarget{lemma-2}{%
\label{lemma-2}}

If \(f\) and \(g\) are concave functions, and \(f\) is non-decreasing, then \(f \circ g\) is also concave and non-decreasing.
\end{lemma}

\begin{proof}

\hypertarget{proof-2}{%
\label{proof-2}}
\(f \circ g\) is clearly non-decreasing. We hence show that it's concave: Given any \(x, y \in \mathbb R\) and \(\lambda\in[0,1]\), we have that

\[\begin{aligned}
& g(\lambda x + (1 - \lambda)y) &\geq& \lambda g(x) + (1-\lambda) g(y)\\
\therefore& f(g(\lambda x + (1 - \lambda)y)) &\geq& f(\lambda g(x) + (1-\lambda) g(y)) \qquad{\text{because f is non-decreasing}}\\
&&\geq& \lambda f(g(x)) + (1-\lambda) f(g(y)) \qquad{\text{because f is concave}}
\end{aligned}\]

\end{proof}

\begin{lemma}
\hypertarget{lemma-3}{%
\label{lemma-3}}

For any concave and non-decreasing function \(h\), pair of real numbers \(u > l\), and two sequences \((a_k)_{k=1}^n\) and \((b_k)_{k=1}^n\) of equal length where each \(a_k \leq b_k\), it holds that \[\begin{aligned}
h(a_1 + h(a_2 + \dots h(a_n + u))) - h(a_1 + h(a_2 + \dots h(a_n + l))) \geq
\\ h(b_1 + h(b_2 + \dots h(b_n + u))) - h(b_1 + h(b_2 + \dots h(b_n + l))).
\end{aligned}\]

\end{lemma}
\begin{proof}

\hypertarget{proof-3}{%
\label{proof-3}}
Repeated applications of Lemma \ref{lemma-2} and Corollary \ref{corollary-1.1} give:

\[\begin{aligned}
&h(a_1 + \dots h(a_n + u)) - h(a_1 + \dots h(a_n + l))\\
&\geq h(b_1 + h(a_2 + \dots h(a_n + u))) - h(b_1 + h(a_2 + \dots h(a_n + l)))\\
&\geq h(b_1 + h(b_2 + \dots h(a_n + u))) - h(b_1 + h(b_2 + \dots h(a_n + l)))\\
&\dots\\
&\geq h(b_1 + h(b_2 + \dots h(b_n + u))) - h(b_1 + h(b_2 + \dots h(b_n + l)))\\
\end{aligned}\]

\end{proof}
\hypertarget{example-of-using-lemma-3}{%
\section{Example of using Lemma \ref{lemma-3}}\label{example-of-using-lemma-3}}

We consider the easier problem of showing that \(\arctan(A + \arctan(B + \arctan(C + \dotsc)))\) converges whenever \(A, B, C\) etc. are positive. Non-constructively, we can immediately see that this converges because of the Monotone Convergence Theorem. On the other hand, proving it constructively is somewhat of a challenge.

Assume we've only observed \(A\) and \(B\), but not \(C\), or any term after \(C\). Then the maximum value of the expression is \(\arctan(A + \arctan(B + \arctan(\infty)))\), and the minimum value is \(\arctan(A + \arctan(B + \arctan(0)))\). We'll use the latter expression as an estimate. An upper bound for the resulting error \(\epsilon\) is thus \[\epsilon \leq \arctan(A + \arctan(B + \arctan(\infty))) - \arctan(A + \arctan(B + \arctan(0))).\] We now use Lemma \ref{lemma-3} to derive an upper bound on the error by driving \(A\) and \(B\) to zero: \[\begin{aligned}\epsilon &\leq \arctan(0 + \arctan(0 + \arctan(\infty))) - \arctan(0 + \arctan(0 + \arctan(0)))\\
&= \arctan(\arctan(\arctan(\infty))).
\end{aligned}\] So as we observe \(n\) terms (the \(A\), \(B\), \(C\), \(D\)s etc.) our upper bound for the error is \(\arctan^n(\infty)\) where the superscript \(n\) denotes \(n\)-fold iteration. We are done using Lemma \ref{lemma-3}.

It remains to show that \(\arctan^n(\infty)\) converges to \(0\). To do so, first observe that \(\arctan^n(\infty) = \arctan^{n-1}(\pi/2)\).

We can proceed either by attempting to show that \(\arctan^{n-1}(\pi/2) \sim \sqrt{\frac{3}{2n}}\); or by using a sledgehammer, which we do below:

We apply the following lemma: For any function \(f\) such that for all \(x \neq 0\), we have \(|f(x)| < |x|\), it follows that iterates of \(f\) always converge to \(0\). This claim is equivalent to the Fan Principle (sometimes confusingly called the Fan Theorem), which is a constructively acceptable postulate. Equivalence to the Fan Principle is demonstrated by Proposition 3.3.4 of \cite{diener2018constructive}.

\hypertarget{kappa-notation}{%
\section{Kappa notation}\label{kappa-notation}}

We define an operator \(\kap\) which takes as input an ordinal-indexed sequence, and outputs a number.

The value of \[\kap_i(\alpha_i)\] is defined to be the limit of the sequence \[\alpha_\omega, \sqrt{\alpha_1^{2^1} + \alpha_\omega^{2^1}}, \sqrt{\alpha_1^{2^1} + \sqrt{\alpha_2^{2^2} + \alpha_\omega^{2^2}}}, \sqrt{\alpha_1^{2^1} + \sqrt{\alpha_2^{2^2} + \sqrt{\alpha_3^{2^3} + \alpha_\omega^{2^3}}}}, \dotsc\]

By definition, these approximants are equal to \[\kap_i(\alpha_i(1-[0<i<\omega])), \kap_i(\alpha_i(1-[1<i<\omega])), \kap_i(\alpha_i(1-[2<i<\omega])), \dotsc\] where we use Iverson Bracket notation.

The infinite radicals that can be expressed using this notation are more general than the nested radicals considered by Herschfeld and many others. Namely, the sequence \(\alpha_i\) can be continued ``past infinity''. For this reason, we will refer to this more general class of infinite radicals as \emph{transfinite radicals}.

The definition assumes that the sequence is indexed up to and including the ordinal \(\omega\). It \emph{is} possible to generalise the definition further to allow for arbitrarily large ordinals, but we won't do this for the following reasons:

\begin{itemize}
\tightlist
\item
  The results we prove won't change.
\item
  The definitions may become more confusing.
\item
  We may be asked to provide a constructive theory of the ordinal numbers, which is not the goal of this paper.
\end{itemize}

Some examples of transfinite radicals expressible using this notation include:

\begin{itemize}
\tightlist
\item
  \(\kap_{i}(2) = \sqrt{2^{2^1} + \sqrt{2^{2^2} + \sqrt{2^{2^3} + \dotsb}}}\).
\item
  \(\kap_{i}(x^{[i = \omega]})\) where \(x\) is arbitrary. Here, we raise \(x\) to the power of an Iverson Bracket. By definition, this transfinite radical is the limit of the sequence \(x, \sqrt{1 + x^2}, \sqrt{1 + \sqrt{1 + x^4}}, \sqrt{1 + \sqrt{1 + \sqrt{1 + x^8}}},\dotsc\) We will make two points about this example:

  \begin{itemize}
  \tightlist
  \item
    There is no way of viewing this transfinite radical as an ordinary infinite radical.
  \item
    At some point, we will call this function of \(x\) the \(U\) function.
  \end{itemize}
\end{itemize}

Our operator differs from the one used by \cite{continuedReciprocalRoots} and \cite{kersten1992entwicklungssaetze}: Their notation is \[\mathop{\mathrm{\text{\Large{$K$}}}}_{i}\sqrt{\alpha_i},\] which stands for \(\sqrt{\alpha_1 + \sqrt{\alpha_2 + \sqrt{\alpha_3 + \dotsc}}}\). However, their notation is not appropriate for expressing transfinite radicals.

\hypertarget{radical-capped-implies-radical-converges}{%
\section{Radical capped implies radical converges}\label{radical-capped-implies-radical-converges}}

\begin{lemma}
\hypertarget{lemma-4}{%
\label{lemma-4}}

The function \(x \mapsto x^{1/n}\), \(\mathbb R^+ \to \mathbb R^+\) with \(n > 1\) is a concave function.

\end{lemma}
\begin{proof}

\hypertarget{proof-4}{%
\label{proof-4}}
The second derivative is negative which implies concavity.

\end{proof}
\begin{lemma}
\hypertarget{lemma-5}{%
\label{lemma-5}}

Given a sequence \((\alpha_k)_{k=1}^n\) of non-negative numbers, a pair of numbers \(l < u\), and a natural number \(I\), we have that \(\kap_i([i < I]\alpha_i + [i = I]u) - \kap_i([i < I]\alpha_i + [i = I]l) \leq u - l\).

\end{lemma}
\begin{proof}

\hypertarget{proof-5}{%
\label{proof-5}}
This follows from Lemma \ref{lemma-3}, where \(h(x) = \sqrt{x}\), \(a_k = 0\) and \(b_k=\alpha_k^{2^k}\).

We get that \[\begin{aligned}
\sqrt{\alpha_1^{2^1} + \dots \sqrt{\alpha_n^{2^n} + u^{2^n}}} - \sqrt{\alpha_1^{2^1} + \dots \sqrt{\alpha_n^{2^n} + l^{2^n}}} &\leq \sqrt{0 + \dots \sqrt{0 + u^{2^n}}} - \sqrt{0 + \dots \sqrt{0 + l^{2^n}}}\\
&=u - l
\end{aligned}\]

\end{proof}
\begin{lemma}
\hypertarget{lemma-6}{%
\label{lemma-6}}

The \(\kap\) operator has the following properties:

\begin{enumerate}
\def\labelenumi{\arabic{enumi}.}
\tightlist
\item
  \(\alpha_i \leq \beta_i \implies \kap_i(\alpha_i) \leq \kap_i(\beta_i)\)
\item
  \(\kap_i(C\alpha_i) = C\kap_i(\alpha_i)\) where \(C\) is any nonnegative constant.
\item
  \(\kap_i(\alpha_i) \geq \alpha_n\) for any particular element of the sequence \(\alpha_n\).
\item
  \(\kap_i(1) = \phi\) where \(\phi\) denotes the Golden Ratio.
\item
  \(\kap_i(\alpha_i[i < n] + \beta_{i - n + 1}[i \geq n]) = \kap_i(\alpha_i[i < n] + [i = n](\kap_j(\beta_j^{2^n}))^{2^{-n}})\). This rule allows us to use recursion.
\item
  \(\kap_i(\alpha_i[i < n] + x[i=n]) = \kap_i(\alpha_i[i < n] + x[i=\omega])\). We will call this the \emph{shift rule}. It allows us to shift the last term to infinity.
\end{enumerate}

\end{lemma}
\begin{proof}

\hypertarget{proof-6}{%
\label{proof-6}}
Obvious.
\end{proof}

\begin{lemma}
\hypertarget{lemma-7}{%
\label{lemma-7}}

If \(\sup_k({\alpha_k}) = M\) then \(\kap_i(\alpha_i) \geq M\).

\end{lemma}
\begin{proof}

\hypertarget{proof-7}{%
\label{proof-7}}
Let \(\epsilon > 0\).

From the fact that \(\sup_k({\alpha_k}) = M\), there is some \(\alpha_N\) such that \(\alpha_N > M - \epsilon\). It follows that \(\kap_i(\alpha_i) \geq \alpha_N > M - \epsilon\). Since \(\epsilon\) is arbitrary, we have that \(\kap_i(\alpha_i) \geq M\).

\end{proof}
\begin{corollary}
\label{corollary-7.1}
If we have a sequence \((M_n)_n\) where each \(M_n = \sup_{k \geq n} \alpha_k\), then for each \(n\) we have a lower bound \(\kap_i(\alpha_i[i < n] + M_n[i = n])\).
\end{corollary}
\begin{proof}

\hypertarget{proof-8}{%
\label{proof-8}}
Let \(\beta_j = \alpha_{j+n-1}\). We have that \(\sup_j\beta_j \leq M_n\). It therefore follows that \(\sup_j\beta_j^{2^n} \leq M_n^{2^{n}}\). Therefore by Lemma \ref{lemma-7}, we have \(\kap_j(\beta_j^{2^n}) \geq M_n^{2^n}\). We finally have that \[\begin{aligned}
\kap_i(\alpha_i) &= \kap_i(\alpha_i[i < n] + \beta_{i - n + 1}[i \geq n])\\
&=\kap_i(\alpha_i[i < n] + [i = n](\kap_j(\beta_j^{2^n}))^{2^{-n}})\\
&\geq \kap_i(\alpha_i[i < n] + [i = n]M_n)\\
\end{aligned}\] \end{proof}
\begin{lemma}
\hypertarget{lemma-8}{%
\label{lemma-8}}

If we have a sequence \((M_n)_n\) where each \(M_n = \sup_{k \geq n} \alpha_k\), then for each \(n\) we have an upper bound \(\kap_i(\alpha_i[i < n] + M_n\phi^{2^{-n}}[i = n])\).

\end{lemma}
\begin{proof}

\hypertarget{proof-9}{%
\label{proof-9}}
\[\begin{aligned}
\alpha_i &\leq \alpha_i[i < n] + M_n[i \geq n] \\
\therefore \kap_i(\alpha_i) &\leq \kap_i(\alpha_i[i < n] + M_n[i \geq n]) \\
&= \kap_i(\alpha_i[i < n] + (\kap_j(M_n^{2^n}))^{2^{-n}}[i = n])\\
&= \kap_i(\alpha_i[i < n] + M_n(\kap_j(1))^{2^{-n}}[i = n])\\
&= \kap_i(\alpha_i[i < n] + M_n\phi^{2^{-n}}[i=n])\\
\end{aligned}\]

\end{proof}

\begin{theorem}
\hypertarget{theorem-sufficient}{%
\label{theorem-sufficient}}

Given the infinite sequence \((\alpha_k)_{k=1}^\infty\) where each \(\alpha_k\) is a non-negative real number, if there exists an infinite sequence \((M_k)_{k=1}^\infty\) such that \(M_n = \sup_{k=n}^\infty \alpha_k\), then \(\kap_i(\alpha_i)\) converges.
\end{theorem}

\begin{proof}

\hypertarget{proof-10}{%
\label{proof-10}}

By Corollary \ref{corollary-7.1}, we have \(\kap_i(\alpha_i) \geq\kap_i(\alpha_i[i < n] + M_n[i = n])\).

By Lemma \ref{lemma-8}, we have \(\kap_i(\alpha_i) \leq\kap_i(\alpha_i[i < n] + M_n\phi^{2^{-n}}[i = n])\).

The difference between upper bound and lower bound is by definition \(\kap_i(\alpha_i[i < n] + M_n\phi^{2^{-n}}[i = n]) - \kap_i(\alpha_i[i < n] + M_n[i = n])\), which by Lemma \ref{lemma-5} is at most \(M_n(\phi^{2^{-n}}-1)\), which clearly goes to zero.

\end{proof}

\hypertarget{the-function-ur}{%
\section{\texorpdfstring{The function \(U(r)\)}{The function U(r)}}\label{the-function-ur}}

\includegraphics{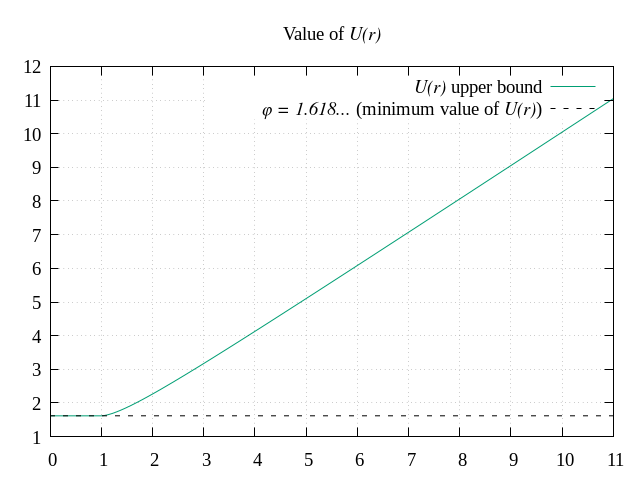} \includegraphics{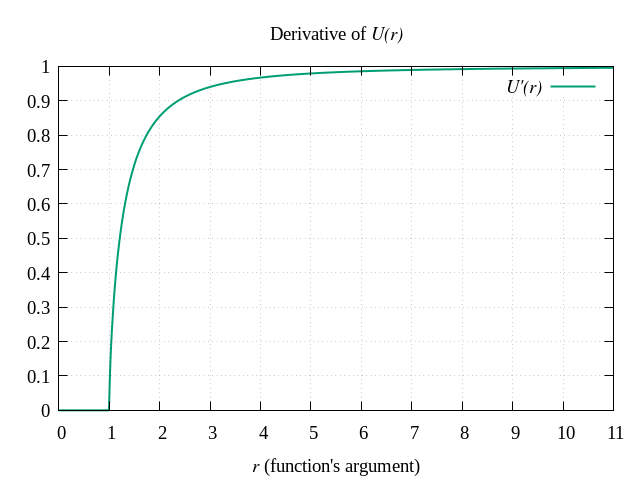}

In this section, we define an interesting function called \(U(r)\), and investigate some of its properties. We will use it to prove Theorem 2. It is defined as \[U(r) = \kap_i([i < \omega] + r[i = \omega])\] It follows from Theorem 1 that the function is well-defined for all \(r \geq 0\).

\begin{lemma}
\hypertarget{lemma-9}{%
\label{lemma-9}}

For \(s \in \mathbb R\) and \(r \in \mathbb R\), if \(s > r \geq 1\), then \(U(s) > U(r)\).

\end{lemma}
\begin{proof}

\hypertarget{proof-11}{%
\label{proof-11}}
If \(r \geq 1\), then Corollary \ref{corollary-7.1} gives us that \(U(r) \leq \kap_i([i < n] + r\phi^{2^{-n}}[i = n])\)

If \(s \geq 1\), then Lemma \ref{lemma-8} gives us that \(U(s) \geq \kap_i([i < n] + s[i = n])\)

Clearly for large enough \(n\), we have that \(s > r\phi^{2^{-n}}\). It follows that for this value of \(n\), we get: \[\begin{aligned}
U(s) &\geq \kap_i([i < n] + s[i = n]) 
\\ &> \kap_i([i < n] + r\phi^{2^{-n}}[i = n]) 
\\ &\geq U(r).
\end{aligned}\]

\end{proof}
\begin{corollary}
\label{corollary-9.1}
\(U:[1,\infty) \to [U(1), \infty)\) is continuous, unbounded, and admits an inverse function.
\end{corollary}
\begin{proof}

\hypertarget{proof-12}{%
\label{proof-12}}
We shall prove Lipschitz continuity: Lipschitz continuity follows from the fact that for every \(n \in \mathbb N\) \[\begin{aligned}
U(s) - U(r) &\leq \kap_i([i < n] + \phi^{2^{-n}}s[i = n]) - \kap_i([i < n] + r[i = n])\\
&\leq \phi^{2^{-n}}s - r \qquad{\text{by Lemma \ref{lemma-5}}}
\end{aligned}\] which implies that \(|U(s) - U(r)| \leq |s - r|\).

The function is unbounded because \(U(r) > r\) (Lemma \ref{lemma-8}).

The existence of an inverse follows from the fact that the function is one-to-one, unbounded, and by the Intermediate Value Theorem.
\end{proof}
\hypertarget{radical-converges-implies-radical-capped}{%
\section{Radical converges implies radical capped}\label{radical-converges-implies-radical-capped}}

\begin{lemma}
\hypertarget{lemma-10}{%
\label{lemma-10}}

If we let \((\beta_i)\) denote the result of rearranging from smaller to larger two adjacent elements of \((\alpha_i)\), then \(\kap_i(\alpha_i) \geq \kap_i(\beta_i)\).

\end{lemma}
\begin{proof}

\hypertarget{proof-13}{%
\label{proof-13}}
This follows from the fact that the square-root is a convex, increasing function that maps nonnegative numbers to nonnegative numbers.

Let \(f(x) = \sqrt{x}\).

It suffices to consider only swapping the first two elements of the sequence \((\alpha_i)\). In other words, we define the sequence \((\beta_i)\) by \(\beta_i = \min(\alpha_1,\alpha_2)[i = 1] + \max(\alpha_1,\alpha_2)[i=2] + \alpha_i[i>2]\), and claim that \(\kap_i(\alpha_i) \geq \kap_i(\beta_i)\).

Finally our claim is that given \(A \leq B\) and arbitrary nonnegative \(x\), that \[f(f^{-1}(A) + f(f^{-2}(B) + x)) \leq f(f^{-1}(B) + f(f^{-2}(A) + x)).\]

We now prove it:

\[\begin{aligned}
A &\leq B\\
f^{-2}(A) &\leq f^{-2}(B) \qquad{\text{because $f$ is an increasing function}}\\
f^{-2}(A) + x&\leq f^{-2}(B) + x\\
f'(f^{-2}(A) + x)&\geq f'(f^{-2}(B) + x) \qquad{\text{because $f'$ is a decreasing function}}\\
f(f^{-2}(A) + x) - f^{-1}(A) &\geq f(f^{-2}(B) + x) - f^{-1}(B)\qquad{\text{by integrating from $0$ to $x$}}\\
f^{-1}(B) + f(f^{-2}(A) + x) &\geq f^{-1}(A) + f(f^{-2}(B) + x) \qquad{\text{by rearranging both sides}}\\
f(f^{-1}(B) + f(f^{-2}(A) + x)) &\geq f(f^{-1}(A) + f(f^{-2}(B) + x)) \qquad{\text{because $f$ is an increasing function}}\\
\end{aligned}\]

\end{proof}
Here, we are about to generalise the definition of a transfinite radical. We've avoided doing this up until now in order to keep things simple. All the results we've proved previously carry through.

\begin{definition}
For each finite set of ordinals \(\Omega=\{\iota_1, \iota_2, \dotsc, \iota_n\}\) where \(\iota_1 < \iota_2 < \dotsc < \iota_n\), let \(\kap_{i \in \Omega} (\alpha_i)\) denote \(\sqrt{\alpha_{\iota_1}^{2^1} + \sqrt{\alpha_{\iota_2}^{2^2} + \dotsc \sqrt{\alpha_{\iota_n}^{2^n}}}}\).
\end{definition}

Clearly, \(\kap_i(\alpha_i)\) is convergent iff \(\kap_{i \in \Omega}(\alpha_i)\) is \emph{Cauchy} in \(\Omega\). What this means is that for any \(\epsilon>0\) there exists a finite set of ordinals \(H\) such that for all finite sets of ordinals \(K \supset H\), \(\lvert\kap_{i \in K}(\alpha_i) - \kap_{i \in H}(\alpha_i)\rvert < \epsilon\).

\begin{definition}
For each finite set of ordinals \(\Omega\), let \(M_\Omega = \max_{k \in \Omega} \alpha_k\).
\end{definition}

\begin{lemma}
\hypertarget{lemma-11}{%
\label{lemma-11}}

Given the infinite sequence \((\alpha_k)_{k=1}^\infty\), if the radical \(\kap_i(\alpha_i)\) is convergent, then there exists an \(M \in \mathbb R\) such that \(M = \sup_{k=1}^\infty \alpha_k\).

\end{lemma}
\begin{proof}

\hypertarget{proof-14}{%
\label{proof-14}}
The following argument will be somewhat informal. This is so as to not confuse the reader.

Let \(\epsilon > 0\). Convergence means that we have a finite set of ordinals \(H\) such that for any finite set of ordinals \(K \supset H\), \(\kap_{i\in K}(\alpha_i) - \kap_{i \in H}(\alpha_i) < \epsilon\).

Consider an \(H\) for which the above is true. Consider any \(k \in \mathbb N\) such that \(\alpha_k > M_H\), and let \(K = H \cup \{\alpha_k\}\).

\(\kap_{i \in K}(\alpha_i)\) can be written explicitly as \(\sqrt{\alpha_1^{2^1} + \sqrt{\alpha_2^{2^2} + \dotsc\sqrt{\alpha_n^{2^n}}}}\). We're assuming that \(H=\{1,\dotsc,n\}\setminus\{k\}\) and \(k \leq n\), but this is without any loss of generality. Hence we have: \[\epsilon > \sqrt{\alpha_1^{2^1} + \sqrt{\alpha_2^{2^2} + \dotsc\sqrt{\alpha_n^{2^n}}}} - \kap_{i \in H}(\alpha_i)\] Let's say, for instance, that \(k = 2\). We will write \(\alpha_k\) as \(M_K\) to distinguish it. We thus have: \[\epsilon > \sqrt{\alpha_1^{2^1} + \sqrt{M_K^{2^2} + \dotsc\sqrt{\alpha_n^{2^n}}}} - \kap_{i \in H}(\alpha_i)\] We now shift \(M_K\) to the last index. By Lemma \ref{lemma-10}, this shrinks the lower bound: \[\epsilon > \sqrt{\alpha_1^{2^1} + \sqrt{\alpha_3^{2^2} + \dotsc\sqrt{\alpha_n^{2^n} + M_K^{2^n}}}} - \kap_{i \in H}(\alpha_i)\] We then drive all the other \(\alpha_i\) to \(M_H\), which they are all at most equal to. This reduces the lower bound further. \[\epsilon > \sqrt{M_H^{2^1} + \sqrt{M_H^{2^2} + \dotsc\sqrt{M_H^{2^n} + M_K^{2^n}}}} - \kap_{i \in H}(M_H)\] We factorise out \(M_H\): \[\epsilon > M_H\left(\sqrt{1 + \sqrt{1 + \dotsc\sqrt{1 + (M_K/M_H)^{2^n}}}} - \kap_{i \in H}(1)\right)\] We now switch to using notation: \[\begin{aligned}
\epsilon &> M_H\left(\kap_i([i < n] + (M_K/M_H)[i=n]) - \kap_i([i < n])\right)\\
&= M_H\left(\kap_i([i < n] + (M_K/M_H)[i=\omega]) - \kap_i([i < n])\right) \qquad{\text{shift rule from Lemma \ref{lemma-6}}}\\
&\geq M_H\left(\kap_i([i < \omega] + (M_K/M_H)[i=\omega]) - \kap_i([i < \omega])\right) \qquad{\text{using Lemma \ref{lemma-3} to drive $0$ terms to $1$}}\\
&=M_H(U(M_K/M_H) - U(1))
\end{aligned}\] We extract out the inequality \(\epsilon > M_H(U(M_K / M_H) - U(1))\), and rearrange it to: \[\alpha_K < M_H \cdot U^{-1}\left(\frac{\epsilon}{M_H} + U(1)\right)\] We thus conclude that the maximum value \(M\) exists, and can be bounded as: \[M \in \left[M_H,M_H \cdot U^{-1}\left(\frac{\epsilon}{M_H} + U(1)\right)\right].\] \end{proof}

\begin{theorem}
\label{theorem-necessity}
Given the infinite sequence \((\alpha_k)_{k=1}^\infty\) where each \(\alpha_k\) is a non-negative real number, if the transfinite radical \(\kap_i(\alpha_i)\) converges, then there exists an infinite sequence \((s_k)_{k=1}^\infty\) such that \(s_n = \sup_{k=n}^\infty a_k^{2^{-k}}\).
\end{theorem}

\begin{proof}
\hypertarget{proof-15}{%
\label{proof-15}}
We construct \(s_n\) by applying Lemma \ref{lemma-11} (which comes with a constructive formula!) to the sequence \(\beta_i = \alpha_i[i \geq n]\).

\end{proof}
\hypertarget{constructivity-of-reformulations-of-herschfelds-theorem}{%
\section{Constructivity of reformulations of Herschfeld's Theorem}\label{constructivity-of-reformulations-of-herschfelds-theorem}}

My constructive statement of Herschfeld's Theorem is given by Theorem 1 and Theorem 2, which combined say:

\begin{quote}
\textbf{Constructive HCT}: An infinite radical \(\sqrt{a_1 + \sqrt{a_2 + \dots}}\) where each \(a_k\) is a non-negative real number converges if and only if there is a sequence \(s_n\) such that \(s_n = \sup_{k=n}^\infty a_k^{2^{-k}}\).
\end{quote}

Herschfeld's original statement of his theorem is somewhat simpler:

\begin{quote}
\textbf{Strong HCT}: An infinite radical \(\sqrt{a_1 + \sqrt{a_2 + \dots}}\) where each \(a_k\) is a non-negative real number converges if and only if there is a constant \(M\) such that \(M = \limsup_{k=1}^\infty a_k^{2^{-k}}\).
\end{quote}

We shall break Strong HCT into a conjuction of two statements, both converses of each other.

\begin{quote}
\textbf{Strong HCT a}: Given a sequence \((a_k)\) where each \(a_k\) is a non-negative real number, if there is a constant \(M\) such that \(M = \limsup_{k=1}^\infty a_k^{2^{-k}}\), then \(\sqrt{a_1 + \sqrt{a_2 + \dots}}\) converges.
\end{quote}

\begin{quote}
\textbf{Strong HCT b}: Given a sequence \((a_k)\) where each \(a_k\) is a non-negative real number, if \(\sqrt{a_1 + \sqrt{a_2 + \dots}}\) converges, then there is a constant \(M\) such that \(M = \limsup_{k=1}^\infty a_k^{2^{-k}}\).
\end{quote}

We will now show the ``issues'' with Strong HCT b, by quoting the statement WLPO \cite{bishop1967}.

\begin{quote}
\textbf{WLPO} (\emph{Weak Limited Principle of Omniscience}): For any infinite binary sequence \((b_n)\), either all elements of the sequence are \(1\) or not all elements are \(1\).
\end{quote}

The statement WLPO is universally accepted to be non-constructive. It turns out that Strong HCT b actually implies it.

\begin{theorem}
\label{strong-hct-implies-lpo}
Strong HCT b implies WLPO.

\end{theorem}
\begin{proof}

\hypertarget{proof-16}{%
\label{proof-16}}
Consider an infinite binary sequence \((b_n)_{n \in \mathbb N}\). We define a new sequence \((c_n)\) of real numbers, such that

\[c_n = \begin{cases}
1,& b_n = 1,\\
\phi,& b_{n-1}=1\text{ and } b_n = 0\\
0,& \text{otherwise}
\end{cases}\]

The infinite radical given by \(\sqrt{c_1 + \sqrt{c_2 + \dots}}\) clearly converges to \(\phi\). This means that \emph{we do know} what the radical converges to. We can even bound the rate of convergence. But according to Strong HCT b, this means that we can determine the value of \(M = \limsup_k c_k^{2^{-k}}\).

Either \(M < 1\) or \(M > 0\).

If \(M < 1\) then it can't be that all terms of \(b_n\) are \(1\).

If \(M > 0\) then all terms of \(b_n\) are \(1\).

We have WLPO.

\end{proof}
This shows that our formulation of Constructive HCT, though strange, is necessary to be in the form it is.\footnote{Originally, it was intended for this section to discuss different interpretations of the statement ``\(M = \limsup_{k=1}^\infty a_k^{2^{-k}}\)''. The intention was to show that for an overly strong definition of the \(\limsup\) operator, \emph{Strong HCT b} would be equivalent to LPO; and for an overly weak definition of \(\limsup\), \emph{Strong HCT a} would be equivalent to Markov's Principle. Markov's Principle is a statement considered mildly acceptable to constructivists, but which cannot be derived from the rules of constructive logic alone. This discussion was eventually omitted because I wasn't sure how to write it in a clear way.}

\textbf{Acknowledgements}: The original version of this paper was very different to this one. The change is largely because of e-mail conversations between me and Fred Richman. These conversations challenged me to reformulate the theorem in a way that I could convincingly call it the ``Constructive Herschfeld Theorem''. The result of this was a completely new formulation of the theorem, and a much better strategy for proving it. I'd like to extend my thanks to Fred for his patience.

\bibliography{constructive-herschfeld}{}
\bibliographystyle{plain}

\end{document}